\documentclass[12pt]{amsart}
\usepackage{amscd}
\usepackage{verbatim}
\usepackage{amssymb, amsmath, amsthm, amscd,ifthen}
\usepackage{graphicx}

\usepackage{amsmath,amssymb,amscd,}
\usepackage[all]{xy}

\usepackage{amssymb}

\newtheorem{thm}{Theorem}[section]
\newtheorem{prp}[thm]{Proposition}

\newtheorem{cor}[thm]{Corollary}

\theoremstyle{definition}

\title{Cross-ratios of quadrilateral linkages}

\author{Giorgi Khimshiashvili, Dirk Siersma}

\date{}


\begin{document}
\begin{abstract}
We discuss the cross-ratio map of planar quadrilateral linkages, also
in the case when one of the links is telescopic. Most of our results are valid
for a planar quadrilateral linkage with generic lengths of the sides. In particular,
we describe the image of cross-ratio map for quadrilateral linkage and planar robot
$3$-arm. 
\end{abstract}

\maketitle \setcounter{section}{0}

\section{Introduction}

We deal with quadrilaterals in Euclidean plane $\mathbb{R}^2$ with coordinates $(x,y)$ identified
with the complex plane $\mathbb{C}$ with coordinate $z = x + \imath y$. Given such a quadrilateral $Q$
we define {\it cross-ratio of $Q$} as the cross-ratio of the four complex numbers representing its
vertices in the prescribed order. Using complex numbers in the study of polygons has a long
tradition (see, e.g., \cite{ahl}, \cite{ber}). We present several new developments concerned with
the above notion of cross-ratio of quadrilateral.

The main aim of this paper is to investigate the values of cross-ratio in certain
 families of planar quadrilaterals. Two types of such families are discussed: (1) the 1-dimensional moduli
spaces of quadrilateral linkage \cite{code} and (2) the 2-dimensional moduli spaces of planar robot arms.

In the first part of this paper we deal with quadrilateral linkages (or {\it $4$-bar mechanisms} \cite{gine}).
In spite of apparent simplicity of these objects their study is related to several deep results of algebraic
geometry and function theory, in particular, to the theory of elliptic functions and Poncelet Porism \cite{dui}.
Comprehensive results on the geometry of planar $4$-bar mechanisms are presented in \cite{gine}. Some recent
results may be found in  \cite{dui}, \cite{khi2},\cite{khsi}.

We complement results of \cite{gine} and \cite{khi2} by discussing several new aspects which emerged in course
of our study of extremal problems on moduli spaces of polygonal linkages (cf. \cite{khi1},\cite{khi2}, \cite{khpasizh},
\cite{khsi}, \cite{khsi2} ). In this context it is natural to consider polygonal linkage as a purely mathematical
object defined by a collection of positive numbers and investigate its moduli spaces \cite{code}.
In this paper we deal with quadrilateral linkages and planar moduli spaces.

Two types of quadrilateral linkages are considered: (1) {\it conventional} quadrilateral linkages with the fixed lengths
of the sides, and (2) quadrilateral linkages
with one {\it telescopic link} \cite{code}. Obviously, the latter concept is equivalent to the
so-called {\it planar robot $3$-arm} (or planar triple pendulum \cite{khsi2}). To unify and simplify terminology it
is convenient to refer to these two cases by speaking of {\it closed} and {\it open $4$-vertex linkages}.

The necessary background for our considerations is presented in Section 2. We begin with recalling the definition
and basic geometric properties of planar moduli spaces of $4$-vertex linkages (Proposition \ref{moduli4bar}).
With a planar $4$-vertex linkage $Q$ one can associate the {\it cross-ratio map} $Cr_Q$ from its planar moduli
space $M(Q)$ into the extended complex plane $\mathbb{\overline C}$ (Riemann sphere).
 
Our first main result gives a precise description of the image of cross-ratio map for a generic quadrilateral linkage
(Theorem \ref{crossclosed}). It turns out that 
cross-ratio is a stable mapping in the sense of singularity theory and that its image is an arc of a circle or a full circle, depending on the type of the  moduli space.
 This eventually enables us to obtain an analogous result for a planar robot $3$-arm
(Theorem \ref{crossopen}). Here again cross-ratio is a stable map, having only folds (and no cusps) and the image is an annulus. Moreover the Jacobian of cross-ratio is a non-zero multiple of the signed area and the critical points correspond to quadrilaterals and arms with signed area zero.

In conclusion we mention several possible generalizations of and research
perspectives suggested by our results.

{\bf Acknowledgment.} Joint research on these topics was started during the authors' visit to the Abdus Salam International Centre
for Theoretical Physics in June of 2009. The present paper was completed during a "Research in Pairs" session in CIRM (Luminy)
in January of 2013. The authors acknowledge excellent working conditions in both these institutions and useful
discussions with E.Wegert and G.Panina.

\section{Moduli spaces of planar $4$-vertex linkages}

We freely use some notions and constructions from the mathematical theory of linkages,
in particular, the concept of {\it planar moduli space} of a polygonal linkage \cite{code}.
Recall that {\it closed $n$-lateral linkage} $L(l)$ is defined by
a $n$-tuple $l$ of positive real numbers $l_j$ called its side-lengths such that the biggest of
side-lengths does not exceed the sum of remaining ones. The latter condition guarantees the existence
of a $n$-gon in Euclidean plane $\mathbb{R}^2$ with the lengths of the sides equal to numbers $l_j$.
Each such polygon is called a {\it planar realization} of linkage $L(l)$.

Linkage with a {\it telescopic side} is defined similarly but now the last side-length $l_n$
is allowed to take any positive value. For brevity we will distinguish these two cases by speaking
of closed and open linkages.

For a closed or open linkage $L$, its planar configuration space $M(L) = M_2(L)$ is defined as the set of its
planar realizations (configurations) taken modulo the group of orientation preserving isometries of $\mathbb{R}^2$ \cite{code}.
It is easy to see that moduli spaces $M(L)$ have natural structures of compact real algebraic varieties.
For an open $n$-linkage its planar moduli space is diffeomorphic to the $(n-2)$-dimensional torus $T^{n-2}$.
For a closed $n$-linkage with a {\it generic} side-length vector $l$, its planar moduli space is a smooth compact
$(n-3)$-dimensional manifold. As usual, here and below the term "generic" means "for an open dense subset of
parameter space" (in our setting, this is the space $\mathbb{R}_+^n$ of side-lengths).

In particular, a closed $4$-linkage $Q = Q(l)$ is defined by a quadruple of positive numbers $l = (a, b, c, d) \in \mathbb{R}_+^4$.
An open planar $4$-linkage (or planar robot $3$-arm $A=A(l)$) is analogously defined by a triple of positive numbers $l = (a, b, c)\in \mathbb{R}_+^3$
and its planar moduli space is diffeomorphic to the two-torus $T^2$. The complete list of possible topological types of planar moduli
spaces of closed $4$-linkages is also well known (see, e.g., \cite{kami1}).

\begin{prp} \label{moduli4bar}
The complete list of homeomorphism types of planar moduli spaces of a $4$-bar linkages is as follows:
circle, disjoint union of two circles, bouquet of two circles, two circles with two common points,
three circles with pairwise intersections equal to one point.
\end{prp}

Closed linkages with smooth moduli spaces are called non-degenerate. It is well-known
that non-degeneracy is equivalent to the generic condition $a \pm b \pm c \pm d \ne 0$. It excludes aligned configurations. In the sequel we mainly focus on non-degenerate
quadrilateral linkages, but in our study of  robot $3$-arms we will meet also the degenerate quadrilateral linkages.
See the first row of Figure
\ref{fig:MS_CR} for pictures of the moduli spaces.

\section{Cross-ratio map of quadrilateral linkage} 

In this section we use some basic properties of cross-ratio which can be found in \cite{ber}.
Recall that the complex cross-ratio of four points (where no three of them coincide) $p, q, z, w \in \mathbb{C}$ is defined as
\begin{equation}
[p, q; z, w] = \frac{z - p}{z - q} : \frac{w - p}{w - q} = \frac{p - z}{p - w} \cdot \frac{q - z}{q - w}.
\end{equation}
and takes values in $\mathbb{C} \cup \infty = \mathbb{P}^1(\mathbb{C})$. 
Coinciding pairs correspond to the values $0,1,\infty$.\\
Group $S_4$ acts by permuting points so one can obtain up to six values of the cross-ratio
for a given unordered quadruple of points which are related by well-known relations \cite{ber}.
For further use notice also that the value of cross-ratio is real if and only if the four points lie
on the same circle of straight line \cite{ber}.

Consider now a quadrilateral linkage $Q=Q(a,b,c,d)$. Note that no three vertices can coincide. Then, for each planar configuration
$V = (v_1, v_2, v_3, v_4) \in \mathbb{C}^4$ of $Q$, put
\begin{equation}
Cr(V) = \mathrm{Cr}((v_1, v_2, v_3, v_4)) = [v_1, v_2; v_3, v_4] = \frac{v_3-v_1}{v_3-v_2} : \frac{v_4-v_1}{v_4-v_2}.
\end{equation}

This obviously defines a continuous (in the non-degenerate case actually a real-analytic) mapping
$Cr_Q: M(Q) \rightarrow \mathbb{P}^1(\mathbb{C})$. Our main aim in this section is to describe its image
$\Gamma_Q = \mathrm{Im}\,Cr_Q$ which is obviously a continuous curve in $\mathbb P ^1(\mathbb C)$.
Taking into account some well-known properties of cross-ratio and moduli space, one
immediately obtains a few geometric properties of $\Gamma_Q$.

In particular, its image should be symmetric with respect to real axis. If the linkage $Q$
does not have aligned configurations the points of intersection $\Gamma_Q$ with real axis
correspond to cyclic configurations of $Q$. It is known that $Q$ can have no more then four
distinct cyclic configurations which come into complex conjugate pairs \cite{khi1}. Hence
$\Gamma_Q$ can intersect the real axis in no more than two points. In case  $M(Q)$ has two components
then they are complex-conjugate and the image of $Cr$ is equal to the image of each component, which
implies that $\Gamma_Q$ is connected even though $M(Q)$ may have two components.

In further considerations it is technically more convenient to work with another map
$R: M(Q) \rightarrow \mathbb{P}^1(\mathbb{C}) $ defined by the formula
\begin{equation}
R(v_1, v_2, v_3, v_4) = Cr(v_1, v_3, v_2, v_4) = [v_1 , v_3; v_2, v_4] =\frac{v_2-v_1}{v_2-v_3} : \frac{v_4-v_1}{v_4-v_3}
\end{equation}

\vspace{-0.5cm}

\begin{figure}[htbp]
	\centering
		\includegraphics[width=5 cm]{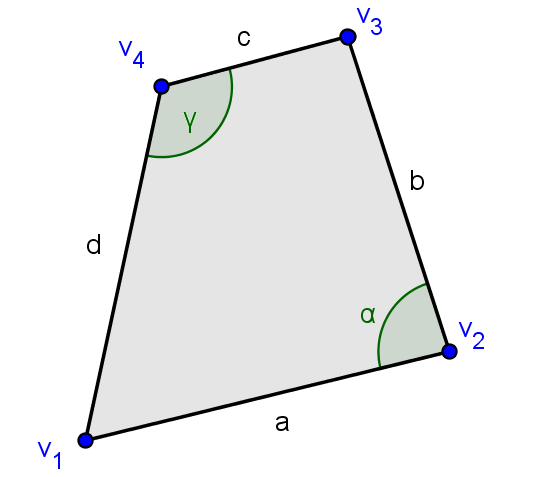}
	\caption{Quadrilateral.}
	\label{fig:Pic4gon}
\end{figure}

From the transformation properties of cross-ratio follows that $$Cr(V) = 1 - R(V).$$
So the properties of $Cr$ can be immediately derived from the properties of $R$.
For brevity we will call $R$ the {\it uniformizer} of $Q$.

The main advantage of $R$ is that, for any configuration $V$ of closed linkage $Q$,
the moduli of numbers $v_{i+1} - v_i$ are constant by its very definition.
Consequently, for any $V\in M(Q)$, one has 
\begin{equation}
\vert R(V) \vert = \frac{ac}{bd}.
\end{equation}
In other words, $R$ maps $M(Q)$ into the circle of radius $\frac{ac}{bd}$
with the center at point $0 \in \mathbb{C}$. Later (in the robot arm case) it is more convenient to consider the chart around $\infty$
and we get a circle with radius  $\frac{bd}{ac}$.
Let 
\begin{equation}
\alpha = \arg \frac{v_3-v_2}{v_1-v_2} \; \;  , \; \;  \gamma = \arg \frac{v_1-v_4}{v_3-v_4}
\end{equation} be the angles at points $v_2$ and $v_4$ in the configuration $V$.

It follows that
\begin{equation}
\arg R(V) = -(\alpha + \gamma), 
\end{equation}

These observations enable us to get a very precise description of the image $\mathrm{Im}\,R$
given in the proposition below. Notice that since a non-singular moduli space is homeomorphic
to a circle or the disjoint union of two circles, one may use the natural orientations of $M(Q)$
and $0 \in \mathbb{C}$ to define the mapping degree of uniformizer map.

\begin{thm} \label{crossclosed}
For a non-degenerate quadrilateral linkage $Q$, the following statements hold:

(1) the image $\mathrm{Im}\,R$ is a subset of the circle
of radius $ac/bd$ centered at the point $0 \in \mathbb{C}$;

(2) the image $\mathrm{Im}\,R$ is connected and symmetric about the real axis
containing the point $\frac{ac}{bd}$;

(3) $R$ is surjective if and only if
$(a+b-c-d)(a-b+c-d)(a-b-c+d) \leq 0.$

(4) the mapping degree of $R$ equals zero and multiplicity at each point does not
exceed two.
\end{thm}

{\bf Proof.} The first two statements follow from the preceding discussion.
The third property can be proved as follows. Take a point $e^{\imath \tau} \in S^1$.
We wish to solve the equation $\mathrm{Arg}\,R(V) = \tau$ with $V\in M(Q)$.
Using the above notation this is equivalent to solving the system
$$\{a^2 + b^2 -2ab \cos \alpha = c^2 + d^2 - 2cd \cos \gamma \; , \; \alpha + \gamma = - \tau\}.$$
Substituting $\cos \alpha = \cos(\tau + \gamma)$  we get
$$a^2 + b^2 -2ab \cos \tau \cos \gamma + 2ab \sin \tau \sin\gamma \; - c^2 -d^2 + 2cd \cos \gamma = 0.$$
From this one easily obtains equation of the form
\begin{equation}
A \sin \gamma + B \cos \gamma = C,
\end{equation}
where $A = 2ab \sin \tau, \; B = -2ab \cos \tau + 2cd, \; C = a^2 + b^2 - c^2 - d^2.$
Now it is easy to see that this equation may have $0$, $1$ or $2$ solutions in $[0, 2\pi]$
depending on the sign of expression 
$$F_{\tau} = A^2 + B^2 - C^2 = 4 a^2 b^2 + 4 c^2 d^2- (a^2+b^2-c^2-d^2)^2 - 8abcd \cos \tau .$$ 
Namely, there are no solutions
if $F_{\tau} < 0$, one solution if $F_{\tau} = 0$, and two solutions if $F_{\tau} > 0$.

It is now easy to conclude that if solution exists for certain $\tau \in [0, \pi]$ then
it exists for any $\sigma \in [0, \pi], \sigma > \tau$ because in this case
$$F_{\sigma}(a,b,c,d) \geq F_{\tau}(a,b,c,d) \geq 0.$$
Hence surjectivity takes place if and only if the point with argument $0$ is in the image of $R$.
Notice that $$F_{0}(a,b,c,d) = -(a+b-c-d)(a-b+c-d)(a-b-c+d)(a+b+c+d).$$
Thus surjectivity is equivalent to $F_{0}(a,b,c,d) \ge 0$ which differs from the criterion
of (3) only by a negative factor $-(a+b+c+d)$. So property (3) is proved.
Property (4) follows form the symmetry of $R$ with respect to the real axis, which
completes the proof of proposition.
\qed

In the non-degenerate case we have:
\begin{cor}
The image of $Cr$ is a conjugation-invariant arc of the circle of radius $ac/bd$ centered
at the point $1 \in \mathbb{C}$.
\end{cor}

This is immediate in view of the relation between $Cr$ and $R$.

\begin{cor}
Cross-ratio map of $Q$ is surjective if and only if $Q$ has a self-intersecting cyclic configuration.
\end{cor}

This follows from the above proof since the argument of $R$ of self-intersecting
cyclic configuration is equal to $0$.

\begin{cor}
Cross-ratio map of $Q$ is surjective if and only if its planar moduli space has two components.
In other words, surjectivity of cross-ratio map is a topological property.
\end{cor}

Indeed, it was shown in \cite{khi1} that a self-intersecting cyclic configuration exists if
and only if the moduli space has two components. Notice that these observations
yield a simple criterion of connectedness of the moduli space.

\begin{cor}
The moduli space is connected if and only if $$(a+b-c-d)(a-b+c-d)(a-b-c+d) \geq 0.$$
\end{cor}

Notice also that, in non-degenerate case, $R(Q)$ is a smooth mapping between two compact one-dimensional
manifolds. We compute its differential with respect to an angular parameter on $M_2(Q)$
and identify its critical points as quadrilaterals  $V$ with signed area equal to zero.
Signed area was defined and it properties were studied in \cite{khi1}.
Moreover we describe the global behaviour as follows:

\begin{thm} \label{crossclosed2}
For a non-degenerate quadrilateral linkage $Q$, the following statements hold:

(5) If $M(Q)$ consists of one component then cross-ratio is a stable mapping with exactly 2 fold points. The image is an arc of a circle,

(6) If $M(Q)$ consists of two components then cross-ratio is a stable mapping, has no singularities and maps each circle bijectively to the image circle.

\begin{proof}
We use Lagrange multipliers for the function  $\mathrm{arg} R(V) = - (\alpha+ \gamma)$ with respect to
\begin{equation} 
g(\alpha,\gamma) =a^2 + b^2 -2ab \cos \alpha - c^2 - d^2 + 2cd \cos \gamma = 0 .
\label{eq:implicit} 
\end{equation}
The critical points of $\mathrm{arg} R(V)$ are given by:
\begin{equation}
2ab \sin \alpha + 2cd \sin \gamma = 0 .
\label{eq:Lmult}
\end{equation}
 This is the condition that the signed Area ($sA$) of the 
quadrilateral is zero ! In \cite{khi1} it is shown, that $sA$ has exactly two critical points on each component and that in the 2-component case it never takes the value 0. It follows that in the 1-component case there are precisely 2 quadrilaterals V with $sA(V) = 0$.

Also the second derivative can  be computed by the Lagrange multipliers method, following \cite{hare}.

The main ingredient for a function $f(x_1,x_2)$ and an equation $g(x_1,x_2)=0$ is the Hessian matrix $H = (\frac{\partial^2 f}{\partial x_i \partial x_j} - \lambda  \frac{\partial^2 g}{\partial x_i \partial x_j})$, where $\lambda$ is defined by $\mathrm{ grad} f = \lambda \; \mathrm{grad} g$.
Evaluate this only on vectors in the tangent space to $g(x_1,x_2)=0 $ at a critical point.
We have in our case (taking into account condition (\ref{eq:Lmult})):
 $$\lambda^{-1} = 2ab \sin \alpha \; \;  ; \; \;
 H = \lambda
\left(\begin{array}[pos]{cc}
	2ab \cos \alpha & 0 \\
	0               & 2cd \cos \gamma \\
\end{array}\right).
$$
The tangent space is generated by $ w^t=-2ab \sin \alpha \; (1,-1)$. The final result for the second derivative is:
$w^t H w = -2 ab \sin \alpha (2ab\cos \alpha - 2cd \cos \gamma)$.
This is zero as soon as  $ 2ab \cos \alpha = 2cd \cos \gamma $. Combining this with condition (\ref{eq:Lmult}) it follows that  $(\alpha, \gamma) \in \{  (0,0), (0,\pi), (\pi,0),(\pi,\pi)\}$. These are the aligned configurations, which are degenerate.
 
So the second derivative is non-zero in the critical points of $\mathrm{arg} R(V) $. By Morse lemma this is enough to conclude that each critical point is a fold.

\end{proof} 
\end{thm}
\begin{figure}[h]
	\centering
		\includegraphics[width=15cm]{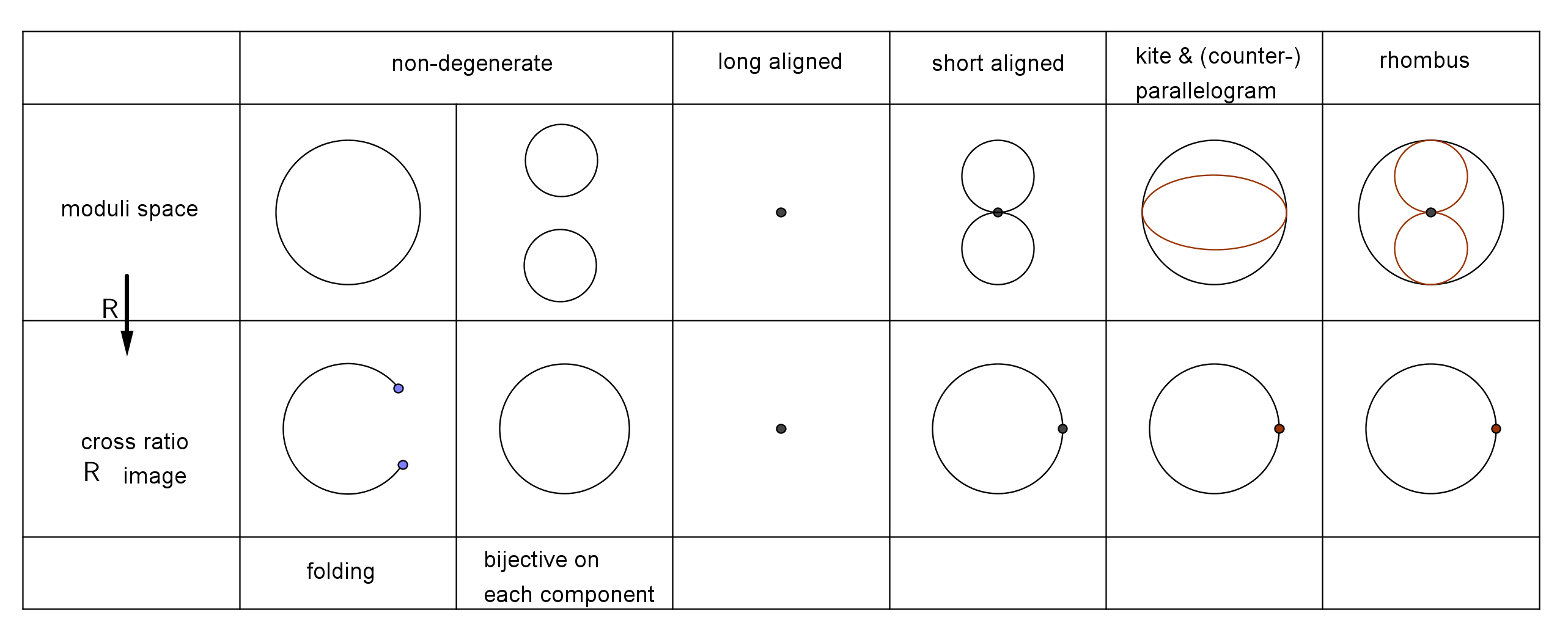}
	\caption{Moduli spaces of quadrilateral linkages and the cross-ratio images}
	\label{fig:MS_CR}
\end{figure}

Next we give the description of the image of $R$ if linkage $Q(a,b,c,d)$ is not generic.
We will use this in the section about robot arms.

\subsection{Long aligned} In this case the length of one edge is equal to the sum of the three others. We have
$$F_{\pi} = (a-b-c-d)(a+b+c-d)(a-b+c+d)(a-c+b+d) = 0.$$
$\tau=0$ is the only possibility, 
the moduli space is a point and the image is also one point.

\subsection{Short aligned} In this case the sum of the lengths of two sides is equal to the sum of the two others.
We have
$$F_{0} = - (a+b-c-d)(a-b+c-d)(a-b-c+d)(a+c+b+d) = 0.$$ Consequently $R$ is surjective.
 When when we are not in the cases \ref{ss:kite}, \ref{ss:pgm} or \ref{ss:romboid} the moduli space  is a bouquet of two circles and the uniformizer $R$ maps each of the two circles onto a full circle. The wedge point is mapped to the intersection of the circle with the positive real axis.

\subsection{Kite} \label{ss:kite}
When  $a=b$ and $c=d$ we have a moduli space, which consists of two circles having two points in common. $R$ maps
one circle $2:1$ (with degree $0$) onto the image circle and the other circle collapses to the point on the positive real axis.

\subsection{Parallelogram and counter-parallelogram}\label{ss:pgm}
When  $a=c$ and $b=d$ we have a moduli space, which consists of two circles having two points in common. $R$ maps
one circle (corresponding to the parallelograms) $2:1$ (with degree $0$) onto the image circle and the other circle (corresponding to the counter-parallelograms) collapses to the point on the positive real axis.

\subsection{Rhomboid}\label{ss:romboid}
When $a=b=c=d$  the moduli space consists of three circles having pairwise a point in common  \cite{code}. Note that $R$ maps
one circle $2:1$ (with degree $0$) onto the image circle and the other two circles collapse to the point on the positive real axis.

\section{Cross-ratio map of robot 3-arm}

The cross-ratio is defined as a map $Cr: M(A) \rightarrow \mathbb{P}^1(\mathbb{C)}$ to the Riemann sphere. 
Only if $M(A(l))$ does contain configurations with coinciding vertices  $Cr$ attains the value $\infty$. This happens besides non-generic cases only if the arm forms a triangle.
Since $M(A)$ is diffeomorphic to $T^2$ we may ask a number of natural questions about the behavior of $Cr$ as a mapping  between 2-dimensional manifolds.
In particular, in the spirit of Whitney's results on stable mappings
 (see, e.g. \cite{ber}) one can wonder if $Cr$ is
stable in the sense of singularity theory: having only folds and cusps as singularities.

As before we work below with the uniformizer $R(Z) = 1 - Cr(Z).$

\begin{thm}
The cross-ratio map for open linkages is a stable mapping with folds only.
\end{thm}
\begin{proof} 
We show this in several steps. First we
consider the $4$-bar linkage $Q_t$ obtained by adding to the arm $Z$ a fourth side of length $t$ and take this $t$ as one of the local coordinates on open subsets of the torus $M(A)$. This is possible as long as $Q_t$ is non-degenerate. Avoid $a \pm b \pm c \pm t = 0,$ where  aligned cases occur.  We can take as other local coordinate $\alpha$ or $\gamma$, which are implicitly related by:
$$ a^2 + b^2 - 2ab \cos \alpha = c^2 + t^2 - 2 c t \cos \gamma.$$ 
Assume we can use $(t,\alpha)$ as local coordinates then $\gamma = \gamma (t, \alpha).$
 We use polar coordinates $(|R^{-1}|,\arg R^{-1})$ on the chart at $\infty$ . Now $$|R^{-1}(Z)| = \frac{b t}{ac} \;,  \;  \arg R^{-1} = \alpha + \gamma$$ and therefore
 the critical points of $R$ are just the union of the critical (=fold) points of each of the closed linkages. 
We next relate this to the criterium for a mapping $F(t,\alpha) = (t, f(t,\alpha))$
to have a fold singularity (cf.,\cite{bro}, p. 74): $\frac{\partial f}{\partial \alpha} = 0 $ and $\frac{\partial^2 f}{\partial^2 \alpha}  \ne 0$ both taken in a point $(t_0,\alpha_0)$. Take for this point the fold point of the quadrilateral $Q_{t_0}$ and it follows that we have indeed a fold for our open linkage. We treat the remaining cases in step 2 and 3.

The second step is to consider the aligned positions.
We choose a complex coordinate on the torus, such that the vertices of the arm are given by $0, a, a+ b e ^{i \phi}, a+ b e^{i \phi}+c e^{i \eta}$. 
This gives
$$ R^{-1} = \tfrac{-b}{ac}(a + b e^{i \phi}+ c e^{i \eta}) e^{i (\phi- \eta)}.$$
For each of the aligned positions we can compute its 2-jet. We take as example
$(\phi,\eta)=(0,0)$. The other cases behave in the same way.
Up to a constant we have for the 2-jet:
$$(\phi,\eta) \rightarrow (a+b+c - \tfrac{1}{2} [a (\phi - \eta)^2 + b (2\phi - \eta)^2 + c \phi^2] \; , \; a(\phi - \eta) + b(2 \phi - \eta) + c \phi) $$
It's singular set is the line $\phi = 0$ and a coordinate transformation brings it in the standard equation of the fold.
Since fold singularities are determined by its 2-jet it follows that also $R^{-1} $ has a fold at the point $(0,0)$.

The third step: In the special case of a closed telescopic arm ($t=0$) (where polar coordinates are not well-defined) we compute in proposition \ref{prp:jac}  the Jacobian in general and show that for $t=0$ it equals to $bc \sin[\eta-\phi]$. This expression is non-zero as soon as the arm is not aligned (non-generic case).

This finishes the proof.

\end{proof}

{\bf Remark:} Note that the other type of stable singularity, the cusp, does not occur!\\

\begin{prp} \label{prp:jac}
The Jacobian of $R^{-1}$ is equal to the signed area of the quadrangle defined by the arm $Z$. The critical points of $R$ and of cross-ratio correspond to the arms with signed area equal to zero.
\begin{proof}
A straightforward computation shows, that the Jacobian is (modulo a non-zero constant) given by:
$$ a b \sin[\phi] + a c \sin[\eta] + b c \sin [\eta - \phi].$$ 
This is precisely twice the signed area of arm $Z$.
\end{proof}
\end{prp}

\begin{cor}
The point $\infty \in \mathbb{P}^1( \mathbb{C})$ is a regular value of uniformizer $R$.
\end{cor}

Next we investigate the shape of the image. We slice with circles.
Fix a number $t \in [0, a+b+c]$
and consider the $4$-bar linkage $Q_t$ obtained by adding to $A$ a fourth side of length $t$. Then the image of $Cr_A$ is simply the union of the the images
$Cr(M(Q_t))$. These are arcs of the circles described in Theorem \ref{crossclosed2}. \\

\begin{figure}
\centering
\begin{minipage}{.5\textwidth}
  \centering
  \includegraphics[width=.8\linewidth]{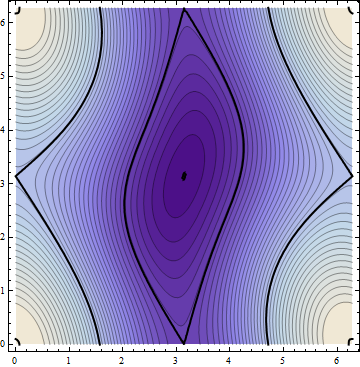}
  \caption{ $t$-levels; \newline case $a > b+c$}
	\label{fig:SpreadLevel1}
\end{minipage}%
\begin{minipage}{.5\textwidth}
  \centering
  \includegraphics[width=.8\linewidth]{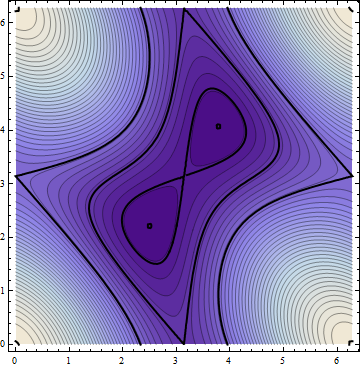}
  \caption{ $t$-levels; \newline case $a < b+c$}
  \label{fig:SpreadLevel2}
\end{minipage}
\end{figure}

We consider the following two cases:
\begin{itemize}
\item[i.] $M(A)$ contains no closed configurations,
\item[ii.]$M(A)$ contains a closed configuration (triangle).
\end{itemize}
We exclude non-generic arms. For sake of presentation we assume $a > b > c$. Other cases behave similar.
So we distinguish now between:
\begin{itemize}
\item[i.] $a > b + c $
\item[ii.] $a < b+c $
\end{itemize} 
The topology of a slice changes at critical values of $t$ (seen as function $M(A) \to \mathbb{R}$).  Level curves are shown in Figures \ref{fig:SpreadLevel1} and \ref{fig:SpreadLevel2}. According to \cite{kami1} these are exactly the aligned positions (where Morse indices follow from the combinatorics) and in the second case also the value $0$.

\begin{figure}[h]
	\centering
		\includegraphics[width=0.85\textwidth]{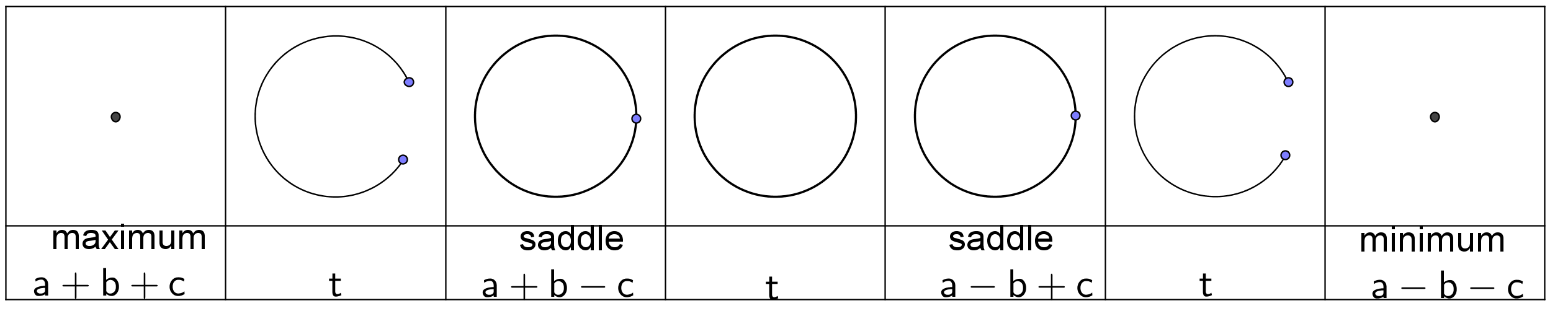}
	\caption{Movie of images of $R$ in case $a > b+c$}
	\label{fig:movie1}
\end{figure}

\begin{figure}[h]
	\centering
		\includegraphics[width=1.1\textwidth]{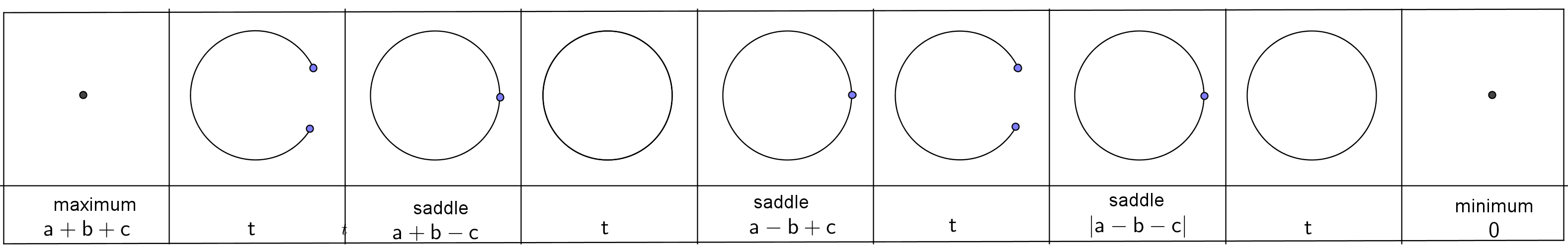}
	\caption{Movie of images of $R$ in case $a < b+c$}
	\label{fig:movie2}
\end{figure}

In case i. we have a Morse function with one maximum, two saddles and a minimum; see the ``movie'' of R-images in figure \ref{fig:movie1}. 
In case ii. (see figure \ref{fig:movie2}) there appears an extra saddle and we end up with two minima, which correspond to $t=0$ (two conjugate triangles). Pictures of the image of $R$ are shown in figures \ref{fig:CRb5c2} and \ref{fig:CRb7c6}.\\

The considerations from section 3 give the following analog of Theorems \ref{crossclosed} and \ref{crossclosed2}:

\begin{thm} \label{crossopen}
For a generic  planar robot $3$-arm  $A(a,b,c)$, the cross-ratio map has degree zero, its image is a conjugation-invariant differentiable annulus and belongs to a disc with radius $a+b+c$.
The cross-ratio map is 2-1 except on the critical set, with image the fold curves.
\end{thm}

\begin{figure}[h]
\centering
\begin{minipage}{.5\textwidth}
  \centering
  \includegraphics[width=.8\linewidth]{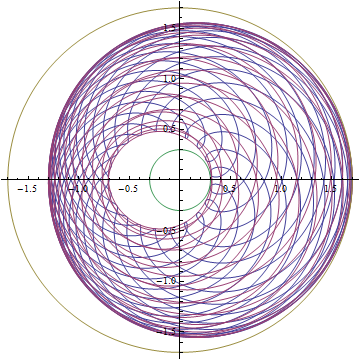}
  \caption{Image of $R$; case $a > b+c$}
	\label{fig:CRb5c2}
\end{minipage}%
\begin{minipage}{.5\textwidth}
  \centering
  \includegraphics[width=.8\linewidth]{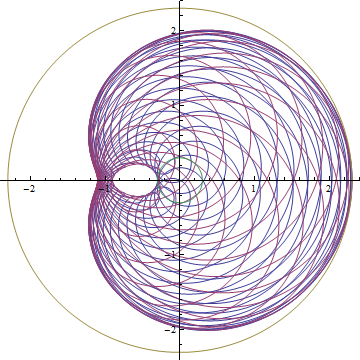}
  \caption{Image of $R$; case $a < b+c$}
  \label{fig:CRb7c6}
\end{minipage}
\end{figure}


\section{Concluding remarks}

First of all, we wish to add that using stereographic projection
one may introduce cross-ratio map for spherical quadrilaterals.
The analogs of Theorems \ref{crossclosed} and \ref{crossopen} follow
in a straightforward way.

It is also interesting to describe the change of cross-ratio
under the action of the so-called Darboux transformation of quadrilateral
linkage \cite{dui}. Taking into account a version of Poncelet Porism for
quadrilateral linkages obtained in \cite{dui} one might hope to get certain
insights concerning the arising discrete dynamical system in the image 
of cross-ratio map. 

In a future paper, by a way of analogy we investigate cross-ratios of one-dimensional
families of the so-called {\it poristic} quadrilaterals arising from Poncelet Porism \cite{fla}.
Analogs of our main results are available for bicentric poristic quadrilaterals 
and poristic quadrilaterals associated with confocal ellipses.

Next, one can also consider cross-ratios of families of quadrilaterals
arising as the centers of circles of Steiner $4$-chains \cite{ber} and
try to describe the image of the corresponding cross-ratio map.

Finally, an analogous line of development arises in connection
with the notion of {\it conformal modulus of a quadrilateral} \cite{ahl}.
In particular, one can try to describe the image and behavior of
conformal modulus for families of poristic bicentric polygons
and confocal ellipses. Developments in this direction will be published
elsewhere.

\medskip

Ilia State University, Tbilisi, Georgia

\smallskip

Utrecht University, The Netherlands

\end{document}